\documentclass[12pt]{amsart}
\usepackage{amssymb}
\usepackage{epsfig}
\usepackage{cite}

\theoremstyle{plain}
\newtheorem{theorem}{Theorem}[section]
\newtheorem{lemma}{Lemma}[section]

\theoremstyle{definition}
\newtheorem{definition}[theorem]{Definition}
\newtheorem{remark}{Remark}[section]

\newtheorem{proposition}[theorem]{proposition}
\numberwithin{equation}{section}
\newtheorem{corollary}{Corollary}[section]

\renewcommand{\theequation}{\the{section}.\arabic{equation}}
\renewcommand{\theequation}{\arabic{section}.\arabic{equation}}

\newcommand{\be}{\begin{equation}}

\newcommand{\ds}{\displaystyle}
\newcommand{\ee}{\end{equation}}
\newcommand{\bea}{\begin{eqnarray}}
\newcommand{\eea}{\end{eqnarray}}

\newcommand{\bl}{\begin{lemma}}
\newcommand{\el}{\end{lemma}}

\newcommand{\bco}{\begin{corollary}}
\newcommand{\eco}{\end{corollary}}

\newcommand{\bc}{\begin{center}}
\newcommand{\ec}{\end{center}}
\newcommand{\bd}{\begin{definition}}
\newcommand{\ed}{\end{definition}}
\newcommand{\ben}{\begin{enumerate}}
\newcommand{\een}{\end{enumerate}}
\newcommand{\bfi}{\begin{figure}}
\newcommand{\efi}{\end{figure}}
\newcommand{\brm}{\begin{remark}}
\newcommand{\erm}{\end{remark}}
\newcommand{\prp}{\begin{proposition}}
\newcommand{\propo}{\end{proposition}}

\newcommand{\R}{\mathbb{R}}
\newcommand{\N}{\mathbb{N}}

\newcommand{\bq}{\begin{quote}}
\newcommand{\eq}{\end{quote}}
\newcommand{\bqu}{\begin{quotation}}
\newcommand{\equ}{\end{quotation}}
\newenvironment{emphit}{\begin{itemize}}{\end{itemize}}
\newcommand{\bemp}{\begin{emphit}}
\newcommand{\eemp}{\end{emphit}}

\begin{document}
\title[Nabla Fractional Operators]
      {Existence and Uniqueness of Solutions of Nabla Fractional Difference Equations Tending to a Nonnegative Constant}

\author[Mert]{Raziye Mert$^1$}
\address{$^1$ Mechatronic Engineering Department, University of Turkish Aeronautical Association, 06790, Ankara, Turkey}
\email{$^1$rmert@thk.edu.tr}

\author[Peterson]{Allan Peterson$^2$}
\address{$^2$ Department of Mathematics, University of Nebraska-Lincoln, Lincoln, NE, 68588-0130,U.S.A.}
\email{$^2$apeterson1@unl.edu}

\author[Abdeljawad]{Thabet Abdeljawad $^3$}
\address{$^3$  Department of Mathematics and General Sciences, Prince Sultan University, P. O. Box 66833, 11586 Riyadh, Saudi Arabia}
\email{$^3$tabdeljawad@psu.edu.sa}

\author[Erbe]{Lynn Erbe$^4$}
\address{$^4$ Department of Mathematics, University of Nebraska-Lincoln, Lincoln, NE, 68588-0130,U.S.A.}
\email{$^4$lerbe@unl.edu}

\date{}

\begin{abstract}
In this paper, we reformulate certain nabla fractional difference equations which had been investigated by other researchers. The previous results seem to be incomplete. By using Contraction Mapping Theorem, we establish conditions under which solutions  exist and are unique and have certain asymptotic properties.

\vspace{.5cm}


\end{abstract}

\maketitle

\section{Introduction}

As a discrete counterpart of classical fractional calculus \cite{MR2, podlubny, Samko, Kilbas}, in recent years, discrete fractional calculus (DFC) has been the  focus of large  number of mathematicians (\cite{Miller,taa,Gray,Nabla,ThCaputo,TDbyparts,Ferd,Feri,AP4,G1,Nuno,Gnabla,Hein,88,89,83,RTE}). This discretizing  issue makes it possible for numerical analysts to develop
 discrete iterated algorithms that enable them  to obtain more accurate solutions  for discrete fractional initial and boundary value problems. Besides the previously mentioned references in the field of DFC, we refer the reader to the useful well-organized book \cite{CP} and for the right case of nabla and delta type discrete fractional sums and differences and the different integration by parts formulas we refer to \cite{dualC,dualR}. Very recently, the discrete versions of new types of fractional operators with nonsingular kernels and some of  their properties have been studied  (see \cite{TDADEMont,TDROMP,MonotonicityChoas,TQ,TQH,TDADIE}) which added a new trend to DFC.

Throughout this paper, for a real number $a$, we denote $\mathbb{N}_a:=\{a,a+1,\cdots \}.$

\begin{definition}(\cite{CP}) The generalized rising function is defined by
$$t^{\overline{r}}=\frac{\Gamma(t+r)}{\Gamma(t)},$$
for values of $t$ and $r$ so that $t, t+r \notin \{0,-1,-2,\cdots \}.$  We use the convention that if   $t$
is a nonpositive integer, but $t+r$ is not a nonpositive integer, then $t^{\overline{r}}:=0$.
\end{definition}

\begin{definition}(\cite{CP})\label{LT}
Let $f:\N_{a+1}\rightarrow\R$ and $\nu> 0,$ then the $\nu$-th order fractional sum based at $a$ is given by
$$\nabla_{a}^{-\nu}f(t)=\int_{a}^t\frac{(t-\rho(s))^{\overline{\nu-1}}}{\Gamma(\nu)} f(s)\,\nabla s=\frac{1}{\Gamma(\nu)}\sum\limits_{s=a+1}^t(t-\rho(s))^{\overline{\nu-1}} f(s),\quad t\in\N_{a}.$$
\end{definition}

\begin{definition}(\cite{CP})
Let $f:\N_{a+1}\rightarrow\R$, $\nu>0$ and choose $N$ such that
$N-1<\nu\leq N$. Then the $\nu$-th order nabla
fractional difference is defined by
$$\nabla_a ^\nu f(t)= \nabla^N\nabla_a^{-(N-\nu)} f (t),\quad t\in\mathbb{N}_{a+N}.$$
\end{definition}

\begin{lemma}(\cite{CP})\label{power} Let $\nu> 0$ and $\mu>-1$. Then for $t \in {\N}_{a}$, we have

$$\nabla_{a}^{-\nu}(t-a)^{\overline{\mu}} = \ds \frac {\Gamma(\mu+1)}{\Gamma(\mu+\nu+1)} (t-a)^{\overline{\mu+\nu}}.$$
\end{lemma}

\begin{theorem}(\cite{taa}) \label{lbsdandds}
For $\nu >0$ and $f$ defined in a suitable domain $\mathbb{N}_a$, we have
\begin{equation} \label{lbdse}
\nabla_a^{\nu}  \nabla_a^{-\nu}f(t)=f(t),
\end{equation}

\begin{equation} \label{lbsde1}
\nabla_a^{-\nu} \nabla_a^{\nu} f(t)=f(t), \quad when  \quad
\nu\notin\mathbb{N},
 \end{equation}
 and
 \begin{equation} \label{lbsde2}
 \nabla_a^{-\nu}  \nabla_a^{\nu}f(t)=
f(t)-\sum_{k=0}^{n-1}\frac{(t-a)^{\overline{k}}}{k!} \nabla^kf(a),
\quad when  \quad \nu=n \in \mathbb{N}.
\end{equation}
\end{theorem}

From classical fractional calculus \cite{Kilbas}, we recall that  $D^{-\alpha} D^{\alpha} f(t) = f(t)$, where $D^{-\alpha}$ is the Riemann-Liuoville fractional integral operator, is valid  for sufficiently well-behaved functions such as continuous functions. Since discrete functions are continuous, we see that the term $\nabla_a^{-(1-\alpha)}f(t)|_{t=a}$, for $0< \alpha <1$ disappears in (\ref{lbsde1}), with the application of the convention that $\sum_{s=a+1}^{a}f(s)=0$. This supports the fact that  Riemann initial type problems usually make sense for functions not necessarily continuous  at $a$ (left case) so that the initial conditions are given in the form $x(a^+)=\lim_{t\rightarrow a^+}x(t)=x_0$. Since sequences are continuous functions the identity (\ref{lbsde1}),  which is the tool in solving initial value problems, appears without any initial condition. In \cite{taa}, to create an initial condition, the authors shifted the fractional difference operator so that it started at $a-1$.   Namely, we state the following theorem.

\begin{theorem} (\cite{taa})\label{IVP}
Consider the initial value problem
\begin{eqnarray}
\label{e3} \nabla_{a-1}^{\nu}y(t) &=& f(t, y(t)),\quad t\in\mathbb{N}_{a+1},\\
\label{e4} \nabla_{a-1}^{-(1-\nu)}y(t)|_{t=a} &=& y(a) = c,
\end{eqnarray}
where $0 <\nu < 1$ and $a$ is any real number. Then $y$ is a  solution of the  initial value problem (\ref{e3})-(\ref{e4}) if and only if $y$ has the representation
\begin{equation*}\label{int1}
 y(t) = \ds \frac{(t-a+1)^{\overline{\nu-1}}}{\Gamma(\nu)} y(a) +  \nabla_{a}^{-\nu} f(t, y(t)),\quad t\in\mathbb{N}_{a}.
\end{equation*}
\end{theorem}

\begin{theorem} (Contraction Mapping Theorem) (\cite{KP}) Let $(X,||.||)$ be a Banach
space. Assume that $T : X\to X$ is a contraction mapping, that is, there is a real number $\alpha$,
$0\leq\alpha < 1$, such that $||Tx-Ty||\leq\alpha||x-y||$ for all $x,y\in X$. Then T has a unique
fixed point $z$ in $X.$
\end{theorem}

\section{Main Results}

In \cite{abi} in Chapter 3, the author uses the Contraction Mapping Theorem to prove the existence and uniqueness of  solutions  of the fractional difference equations
\begin{equation*}\label{abi1}
\nabla_{a}^{\nu}(p\nabla y)(t)+q(t)y(\rho(t))=f(t),\quad t\in\N_{a+1},
\end{equation*}
and
\begin{equation*}\label{abi2}
\nabla_{a}^{\nu}(p\nabla y)(t)+F(t,y(t))=0,\quad t\in\N_{a+1},
\end{equation*}
where $0<\nu<1,$ $a$ is any real number, $p:\mathbb{N}_{a+1}\to (0,\infty)$, $q:\mathbb{N}_{a+1}\to[0,\infty),$ $f:\mathbb{N}_{a+1}\to\mathbb{R},$ and $F:\mathbb{N}_{a+1}\times[0,\infty)\to [0,\infty),$ tending to a nonnegative constant. But, in the proofs, the used solution representations appeared with a term  which turns out to be zero and hence no dependency on the initial condition was reported.  In order to allow the appearance of an initial condition,  we shall reformulate the above equations as
\begin{equation}\label{SA}
\nabla_{a-1}^{\nu}(p\nabla y)(t)+q(t)y(\rho(t))=f(t),\quad t\in\N_{a+1},
\end{equation}
and
\begin{equation}\label{selfad}
\nabla_{a-1}^{\nu}(p\nabla y)(t)+F(t,y(\rho(t)))=0,\quad t\in\N_{a+1},
\end{equation}
where $0<\nu<1,$ $a$ is any real number.


We now prove the main results  for Equation (\ref{selfad}).

\begin{lemma}\label{lemm103}
Let $p:\mathbb{N}_{a}\to (0,\infty)$ and $F:\mathbb{N}_{a+1}\times\R\to [0,\infty).$ For $M\geq 0,$ define
\begin{equation*}\label{set}
\zeta_{M}:=\{y:\N_{a-1}\to[M,\infty):\,\nabla y(t)\leq 0 \,\forall\, t\in\N_{a}\,\,{\text{and}}\,\,\nabla y(a)=0\}.
\end{equation*}
The forced fractional difference equation (\ref{selfad}) has  a solution $y\in\zeta_{M}$ such that $\displaystyle \lim_{t\to\infty}y(t)=M$ if and only if the summation equation
\begin{equation}\label{sumeqn}
y(t)=M+\sum_{s=t+1}^{\infty}\frac{1}{p(s)}\sum_{\tau=a+1}^{s}\frac{(s-\rho(\tau))^{\overline{\nu-1}}}{\Gamma(\nu)}F(\tau, y(\rho(\tau)))
\end{equation}
has a solution $y$ on $\N_{a-1}.$
\end{lemma}

\begin{proof}
Suppose the fractional difference equation (\ref{selfad}) has a solution $y\in\zeta_{M}$ that  satisfies $\displaystyle \lim_{t\to\infty}y(t)=M.$ Let
$x(t):=(p\nabla y)(t).$ Then $x$ solves the fractional initial value problem
\begin{eqnarray*}
\nabla_{a-1}^{\nu} x (t)&=&-F(t,y(\rho(t))),\quad t\in\N_{a+1},\\
x(a)&=&(p\nabla y)(a).
\end{eqnarray*}
\noindent By Theorem \ref{IVP}, $x$ has the representation
\begin{equation*}
x(t)=\frac{(t-a+1)^{\overline{\nu-1}}}{\Gamma(\nu)}x(a)-\nabla_{a}^{-\nu}F(t,y(\rho(t))),\quad t\in\N_{a}.
\end{equation*}
From $\nabla y(a)=0,$ it follows that
\begin{equation*}\label{rep}
\nabla y(t)=-\frac{1}{p(t)}\nabla_{a}^{-\nu}F(t,y(\rho(t))),\quad t\in\N_{a}.
\end{equation*}
Now summing from $s=t+1$ to $\infty$ and using the fact that $\displaystyle \lim_{t\to\infty}y(t)=M,$ we get
$$M-y(t)=-\sum_{s=t+1}^{\infty}\frac{1}{p(s)}\sum_{\tau=a+1}^{s}\frac{(s-\rho(\tau))^{\overline{\nu-1}}}{\Gamma(\nu)}F(\tau, y(\rho(\tau))),\quad t\in\N_{a-1}.$$
Hence,
$$y(t)=M+\sum_{s=t+1}^{\infty}\frac{1}{p(s)}\sum_{\tau=a+1}^{s}\frac{(s-\rho(\tau))^{\overline{\nu-1}}}{\Gamma(\nu)}F(\tau, y(\rho(\tau))),\quad t\in\N_{a-1}.$$
\noindent Thus $y$ is a solution of the summation equation (\ref{sumeqn}).

Conversely, if $y$ is a solution of the summation equation (\ref{sumeqn}) on $\N_{a-1},$ then
$$y(t)=M+\sum_{s=t+1}^{\infty}\frac{1}{p(s)}\sum_{\tau=a+1}^{s}\frac{(s-\rho(\tau))^{\overline{\nu-1}}}{\Gamma(\nu)}F(\tau, y(\rho(\tau))),\quad t\in\N_{a-1}.$$

\noindent Now by taking the nabla difference on both sides of the last equation, we get that
\begin{equation}\label{follows}
\nabla y(t)=-\frac{1}{p(t)}\sum_{\tau=a+1}^{t}\frac{(t-\rho(\tau))^{\overline{\nu-1}}}{\Gamma(\nu)}F(\tau, y(\rho(\tau))),\quad t\in\N_{a}.
\end{equation}
Hence,
\begin{equation*}
p(t)\nabla y(t)=-\nabla_{a}^{-\nu}F(t,y(\rho(t))),\quad t\in\N_{a}.
\end{equation*}
Taking the $\nu$-th difference based at $a-1$ of both sides, we get
\begin{eqnarray*}
\nabla_{a-1}^{\nu}(p\nabla y)(t)&=&-\nabla_{a-1}^{\nu}\nabla_{a}^{-\nu}F(t,y(\rho(t)))\\
&=&-\nabla_{a-1}^{\nu}\left\{\nabla_{a-1}^{-\nu}F(t,y(\rho(t)))-\frac{(t-a+1)^{\overline{\nu-1}}}{\Gamma(\nu)}F(a,y(\rho(a)))\right\}\\
&=&-\nabla_{a-1}^{\nu}\nabla_{a-1}^{-\nu}F(t,y(\rho(t)))+\nabla_{a-1}^{\nu}\left\{\frac{(t-(a-1))^{\overline{\nu-1}}}{\Gamma(\nu)}F(a,y(\rho(a)))\right\}\\
&=&-F(t,y(\rho(t))),\quad t\in\N_{a+1},
\end{eqnarray*}
which follows from the power rule in Lemma \ref{power}. This implies that
$$\nabla_{a-1}^{\nu}(p\nabla y)(t)+F(t,y(\rho(t)))=0,\quad t\in\N_{a+1}.$$
Hence, $y$ is a solution of the fractional difference equation (\ref{selfad}). We also observe that $y(t)\geq M$ for all $t\in\N_{a-1}$ since $p(t)> 0$ for all $t\in\N_{a}$ and $F(t,u)\geq 0$ for all $(t,u)\in\N_{a+1}\times\R.$ From the expression for  $\nabla y(t)$  given by equation (\ref{follows}), we see that $\nabla y(t)\leq 0$ for all $t\in\N_{a}$ and in particular
$$
\nabla y(a)=-\frac{1}{p(a)}\sum_{\tau=a+1}^{a}\frac{(a-\rho(\tau))^{\overline{\nu-1}}}{\Gamma(\nu)}F(\tau, y(\rho(\tau)))=0
$$ by convention. Thus $y\in\zeta_{M}.$ From the convergence of the series, it follows from equation (\ref{sumeqn}) that $\displaystyle \lim_{t\to\infty}y(t)=M$.
\end{proof}

\begin{remark}\label{complete} It is straightforward to prove that the pair $(\zeta_{M},||.||)$, where $\displaystyle ||y||:=\sup_{t\in\mathbb{N}_{a-1}}|y(t)|,$ is a complete metric space.
\end{remark}

\begin{theorem}\label{cont1}
Assume $F:\mathbb{N}_{a+1}\times[0,\infty)\to [0,\infty)$ satisfies a uniform Lipschitz
condition with respect to its second variable, i.e., there is a constant
$K > 0$ such that $$|F(t,u)-F(t,v)|\leq K |u-v|$$ for all $t\in\mathbb{N}_{a+1}, u,v\in[0,\infty)$ and assume $p:\mathbb{N}_{a}\to (0,\infty)$ and let $(\zeta_{M},||.||)$ be the complete metric space as defined in Remark \ref{complete}. If\\
\noindent (H1)\, the series  $\displaystyle \sum_{s=a+1}^{\infty}\frac{1}{p(s)}\sum_{\tau=a+1}^{s}\frac{(s-\rho(\tau))^{\overline{\nu-1}}}{\Gamma(\nu)}F(\tau, y(\rho(\tau)))$ converges for every $y\in\zeta_{M},$\\
\noindent and\\
\noindent (H2)\,$\displaystyle \beta:=\frac{K}{\Gamma(\nu+1)}\left(\sum_{s=a+1}^{\infty}\frac{(s-a)^{\overline{\nu}}}{p(s)}\right)<1$,\\
then there exists a unique positive solution of the  fractional difference equation (\ref{selfad}) with $\displaystyle \lim_{t\to\infty}y(t)=M$.
\end{theorem}
\begin{proof}
Let $(\zeta_{M},||.||)$ be the complete metric space  as defined in Remark \ref{complete}. Define the mapping $T$ on $\zeta_{M}$ by
\begin{equation*}
(Ty)(t):=M+\sum_{s=t+1}^{\infty}\frac{1}{p(s)}\sum_{\tau=a+1}^{s}\frac{(s-\rho(\tau))^{\overline{\nu-1}}}{\Gamma(\nu)}F(\tau, y(\rho(\tau))).
\end{equation*}
Now, we will  show that $T:\zeta_{M}\to\zeta_{M}$. First note that for all $y\in\zeta_{M}$,  $(Ty)(t)\geq M$ for all $t\in\N_{a-1}$ since $p(t)> 0$ for all $t\in\N_{a}$ and $F(t,u)\geq 0$ for all $(t,u)\in\N_{a+1}\times[0,\infty).$
 Next note that
\begin{equation*}
\nabla (Ty)(t)=-\frac{1}{p(t)}\sum_{\tau=a+1}^{t}\frac{(t-\rho(\tau))^{\overline{\nu-1}}}{\Gamma(\nu)}F(\tau, y(\rho(\tau)))\leq 0,\quad t\in\N_{a},
\end{equation*}
 and $\nabla (Ty)(a)=0$ by convention. Hence, $T$ maps $\zeta_{M}$ into itselt. Furthermore, we will show that $T$ is a contraction mapping. Let $x, y\in \zeta_{M}$ and $t\in\N_{a-1}$ be fixed but arbitrary. Then
\begin{eqnarray*}
|(Tx)(t)-(Ty)(t)| &=&\left|\sum_{s=t+1}^{\infty}\frac{1}{p(s)}\sum_{\tau=a+1}^{s}\frac{(s-\rho(\tau))^{\overline{\nu-1}}}{\Gamma(\nu)}(F(\tau, x(\rho(\tau)))-F(\tau, y(\rho(\tau))))\right|\\
&\leq& \sum_{s=t+1}^{\infty}\frac{1}{p(s)}\sum_{\tau=a+1}^{s}\frac{(s-\rho(\tau))^{\overline{\nu-1}}}{\Gamma(\nu)}|F(\tau, x(\rho(\tau)))-F(\tau, y(\rho(\tau)))|\\
&\leq& K \sum_{s=t+1}^{\infty}\frac{1}{p(s)}\sum_{\tau=a+1}^{s}\frac{(s-\rho(\tau))^{\overline{\nu-1}}}{\Gamma(\nu)}| x(\rho(\tau))- y(\rho(\tau))|\\
&\leq& K||x-y||\sum_{s=t+1}^{\infty}\frac{1}{p(s)}\sum_{\tau=a+1}^{s}\frac{(s-\rho(\tau))^{\overline{\nu-1}}}{\Gamma(\nu)}\\
&=&\frac{K}{\Gamma(\nu+1)}\left(\sum_{s=t+1}^{\infty}\frac{(s-a)^{\overline{\nu}}}{p(s)}\right)||x-y||\\
&\leq&\frac{K}{\Gamma(\nu+1)}\left(\sum_{s=a+1}^{\infty}\frac{(s-a)^{\overline{\nu}}}{p(s)}\right)||x-y||\\
&=&\beta ||x-y||.
 \end{eqnarray*}
So
$$||Tx-Ty||\leq\beta ||x-y||$$
with $\beta<1,$ and hence $T$ is a contraction mapping. By Contraction Mapping Theorem, $T$  has a unique fixed point $y\in\zeta_{M}$.
This fixed point satisfies the summation equation (\ref{sumeqn}), and therefore
by Lemma \ref{lemm103}, it is also a solution of the fractional difference equation (\ref{selfad}) that satisfies $\displaystyle\lim_{t\to\infty}y(t)=M.$
\end{proof}
\begin{remark}\label{comp}
Assume $p:\N_{a}\to(0,\infty)$ satisfies $\displaystyle\sum_{s=a}^{\infty}\frac{1}{p(s)}<\infty$ and define $d:\zeta_{M}\times\zeta_{M}\to[0,\infty)$ by
$$d(x,y):=\sup_{t\in\N_{a-1}}\frac{|x(t)-y(t)|}{w(t)},$$
where
$$w(t):=e^{-\left(\sum_{s=a}^{t}\frac{1}{p(s)}\right)}.$$
Note that $0<w(t)\leq 1$ for all $t\in\N_{a-1}$ and $0<\displaystyle L:=\lim_{t\to\infty}w(t)<1.$ Then the pair $(\zeta_{M},d)$ is a complete metric space.
\end{remark}
\begin{proof}
The proof follows as in the proof of Lemma 3.4.1 in \cite{abi}.
\end{proof}
\begin{theorem}
Assume $F:\mathbb{N}_{a+1}\times [0,\infty)\to [0,\infty)$ satisfies a uniform Lipschitz
condition with respect to its second variable, i.e., there is a constant
$K > 0$ such that $$|F(t,u)-F(t,v)|\leq K |u-v|$$ for all $t\in\mathbb{N}_{a+1}, u,v\in [0,\infty)$ and assume $p:\mathbb{N}_{a}\to (0,\infty)$ and let $(\zeta_{M},d)$ be the complete metric space as defined in Remark \ref{comp}. If\\
\noindent (H1)\, the series  $\displaystyle \sum_{s=a+1}^{\infty}\frac{1}{p(s)}\sum_{\tau=a+1}^{s}\frac{(s-\rho(\tau))^{\overline{\nu-1}}}{\Gamma(\nu)}F(\tau, y(\rho(\tau)))$ converges for every $y\in\zeta_{M},$\\
\noindent and\\
\noindent (H2)\,$\displaystyle \alpha:=\frac{K}{L\Gamma(\nu+1)}\sum_{s=a+1}^{\infty}\frac{(s-a)^{\overline{\nu}}}{p(s)}<1$,\\
then there exists a unique positive solution of the  fractional difference equation (\ref{selfad}) with $\displaystyle \lim_{t\to\infty}y(t)=M$.
\end{theorem}
\begin{proof}
Let $(\zeta_{M},d)$ be the complete metric space  as defined in Remark \ref{comp}. As in the proof of Theorem \ref{cont1}, define
the mapping $T$ on $\zeta_{M}$ by
\begin{equation*}\label{mapping}
(Ty)(t):=M+\sum_{s=t+1}^{\infty}\frac{1}{p(s)}\sum_{\tau=a+1}^{s}\frac{(s-\rho(\tau))^{\overline{\nu-1}}}{\Gamma(\nu)}F(\tau, y(\rho(\tau))).
\end{equation*}
We already know that $T: \zeta_{M}\to\zeta_{M}$. Now, we will prove that $T$ is a contraction mapping. Let $x, y\in \zeta_{M}$ and $t\in\N_{a-1}$ be fixed but arbitrary. Then
\begin{eqnarray*}
\frac{|(Tx)(t)-(Ty)(t)|}{w(t)}&=&\frac{1}{w(t)}\left|\sum_{s=t+1}^{\infty}\frac{1}{p(s)}\sum_{\tau=a+1}^{s}\frac{(s-\rho(\tau))^{\overline{\nu-1}}}{\Gamma(\nu)}(F(\tau, x(\rho(\tau)))-F(\tau, y(\rho(\tau))))\right|\\
&\leq& \frac{1}{w(t)}\sum_{s=t+1}^{\infty}\frac{1}{p(s)}\sum_{\tau=a+1}^{s}\frac{(s-\rho(\tau))^{\overline{\nu-1}}}{\Gamma(\nu)}|F(\tau, x(\rho(\tau)))-F(\tau, y(\rho(\tau)))|\\
&\leq& \frac{K}{w(t)} \sum_{s=t+1}^{\infty}\frac{1}{p(s)}\sum_{\tau=a+1}^{s}\frac{(s-\rho(\tau))^{\overline{\nu-1}}}{\Gamma(\nu)}| x(\rho(\tau))- y(\rho(\tau))|\\
&\leq& \frac{K}{w(t)}\left( \sum_{s=t+1}^{\infty}\frac{1}{p(s)}\sum_{\tau=a+1}^{s}\frac{(s-\rho(\tau))^{\overline{\nu-1}}}{\Gamma(\nu)}w(\rho(\tau))\right)d(x,y)\\
&\leq&\frac{K}{L}\left(\sum_{s=t+1}^{\infty}\frac{(s-a)^{\overline{\nu}}}{\Gamma(\nu+1)p(s)}\right)d(x,y)\\
&\leq&\frac{K}{L}\left(\sum_{s=a+1}^{\infty}\frac{(s-a)^{\overline{\nu}}}{\Gamma(\nu+1)p(s)}\right)d(x,y)\\
&=&\alpha d(x,y).
\end{eqnarray*}
So
$$d(Tx,Ty)\leq\alpha d(x,y)$$
with $\alpha<1,$ and hence $T$ is a contraction mapping. By Contraction Mapping Theorem, $T$  has a unique fixed point $y\in\zeta_{M}$.
This fixed point satisfies the summation equation (\ref{sumeqn}), and therefore
by Lemma \ref{lemm103}, it is also a solution of the fractional difference equation (\ref{selfad}) that satisfies $\displaystyle\lim_{t\to\infty}y(t)=M.$
\end{proof}

The results for Equation (\ref{SA}) are as follows:

\begin{lemma}
Let $p:\mathbb{N}_{a}\to (0,\infty)$, $q:\mathbb{N}_{a+1}\to\R,$ and $f:\mathbb{N}_{a+1}\to\R.$ For $M\geq 0,$ define
\begin{equation*}
\xi_{M}:=\{y:\N_{a-1}\to\R:\,\displaystyle \lim_{t\to\infty}y(t)=M\,\,{\text{and}}\,\,\nabla y(a)=0\}.
\end{equation*}
The fractional difference equation (\ref{SA}) has  a solution $y\in\xi_{M}$ if and only if the summation equation
\begin{equation*}
y(t)=M+\sum_{s=t+1}^{\infty}\frac{1}{p(s)}\sum_{\tau=a+1}^{s}\frac{(s-\rho(\tau))^{\overline{\nu-1}}}{\Gamma(\nu)}(q(\tau)y(\rho(\tau))-f(\tau))
\end{equation*}
has a solution $y$ on $\N_{a-1}.$
\end{lemma}
\begin{proof}
The proof is similar to that of Lemma \ref{lemm103}.
\end{proof}
\begin{theorem}
Let $p:\mathbb{N}_{a}\to (0,\infty)$, $q:\mathbb{N}_{a+1}\to[0,\infty),$ and $f:\mathbb{N}_{a+1}\to\R,$ and let  $M\geq 0$ be a real number. Assume that\\
\noindent (H1)\,$\displaystyle \sum_{s=a+1}^{\infty}\frac{1}{p(s)}\sum_{\tau=a+1}^{s}\frac{(s-\rho(\tau))^{\overline{\nu-1}}}{\Gamma(\nu)}q(\tau)<\infty,$\\
\noindent (H2)\,$\displaystyle\sum_{s=a+1}^{\infty}\frac{1}{p(s)}\sum_{\tau=a+1}^{s}\frac{(s-\rho(\tau))^{\overline{\nu-1}}}{\Gamma(\nu)}|f(\tau)|<\infty$.\\
Then there exists some $b\in\N_a$ so that the fractional difference equation
\begin{equation*}
\nabla_{b-1}^{\nu}(p\nabla y)(t)+q(t)y(\rho(t))=f(t),\quad t\in\N_{b+1},
\end{equation*}
 has a solution $y:\N_{b-1}\to\R$ which satisfies  $\displaystyle \lim_{t\to\infty}y(t)=M.$ 
\end{theorem}

\begin{proof}
The proof is similar to that of Theorem 3.2.2. in \cite{abi} except that  we define $\widetilde{\xi}_{b-1}:=\{y:\N_{b-1}\to\R:\,\displaystyle \lim_{t\to\infty}y(t)=M\,\,{\text{and}}\,\,\nabla y(b)=0\}$, where $b\in\mathbb{N}_a,$ and the supremum norm $||.||$ on $\widetilde{\xi}_{b-1}$ by $\displaystyle ||y||:=\sup_{t\in\mathbb{N}_{b-1}}|y(t)|.$
\end{proof}

\section{Acknowledgements}
This study was supported by The Scientific and Technological Research Council of Turkey while the first author visiting the University of Nebraska-Lincoln. The third author would like to thank Prince Sultan University for funding this work through research group
Nonlinear Analysis Methods in Applied Mathematics (NAMAM) group number RG-DES-2017-01-17.

\end{document}